\documentclass{article}
\usepackage[hidelinks]{hyperref}

\usepackage{amsmath}
\usepackage{amsfonts}
\usepackage{mathrsfs}
\usepackage{amsthm}
\usepackage{amssymb}
\usepackage{mathtools}

\usepackage{fullpage}
\usepackage{color}
\usepackage[capitalise]{cleveref}

\usepackage[backend=bibtex,maxbibnames=10,maxcitenames=4]{biblatex}
\DeclareSourcemap{
  \maps[datatype=bibtex]{
    \map[overwrite]{
      \step[fieldsource=doi, final]
      \step[fieldset=url, null]
      \step[fieldset=eprint, null]
    }
  }
}

\usepackage[margin=20pt,font=small,labelfont=bf]{caption}

\usepackage{tikz}
\usepackage{pgfplots}
\pgfplotsset{
    tick align=outside,
    x grid style={white},
    xmajorgrids,
    y grid style={white},
    ymajorgrids,
    axis line style={white},
    axis background/.style={fill=white!92!black},
    legend style={draw=white, fill=white},
    legend cell align={left}
}

\newtheorem{thm}{Theorem}
\newtheorem{defn}[thm]{Definition}
\newtheorem{lemma}[thm]{Lemma}
\newtheorem{prop}[thm]{Proposition}
\newtheorem{remark}[thm]{Remark}

\newtheorem{question}[thm]{Question}

\newcommand{\T}{\mathbb{T}}
\newcommand{\R}{\mathbb{R}}
\newcommand{\C}{\mathbb{C}}
\newcommand{\Z}{\mathbb{Z}}

\newcommand{\fsL}{\textnormal{L}} 

\newcommand{\conj}[1]{\overline{#1}} 
\newcommand{\ip}[2]{\langle #1,#2\rangle} 
\newcommand{\sip}[2]{\left\langle #1,#2\right\rangle} 
\newcommand{\dd}{\mathrm{d}}
\newcommand{\ii}{\mathrm{i}}
\newcommand{\ee}{\mathrm{e}}
\newcommand{\id}{\mathrm{Id}}

\newcommand{\init}{\mathrm{in}}
\newcommand{\const}{\mathrm{const}}

\newcommand{\sol}[3]{\mathcal{S}_{#1}^{{#2}\to{#3}}}

\title{Finding the jump rate for fastest decay in the Goldstein-Taylor model}
\author{Helge Dietert\footnote{Email: \href{mailto:helge.dietert@imj-prg.fr}{helge.dietert@imj-prg.fr}\newline
    Universit\'e de Paris and Sorbonne Universit\'e, CNRS,
    Institut de Math\'ematiques de Jussieu-Paris Rive Gauche (IMJ-PRG),
    F-75013, Paris, France\newline
    Currently on leave and working at\newline
    Institut f\"ur Mathematik, Universit\"at Leipzig, D-04103 Leipzig, Germany
  } \and Josephine Evans\footnote{Email:
    \href{mailto:josephine.evans@warwick.ac.uk}{josephine.evans@warwick.ac.uk}\newline
    Warwick Mathematics Institute, University of Warwick, UK}}
\begin{document}
\maketitle
\begin{abstract}
  For hypocoercive linear kinetic equations we first formulate an
  optimisation problem on a spatially dependent jump rate in order to
  find the fastest decay rate of perturbations. In the
  Goldstein-Taylor model we show (i) that for a locally optimal jump
  rate the spectral gap is determined by multiple, possible
  degenerate, eigenvectors and (ii) that globally the fastest decay is
  obtained with a spatially homogeneous jump rate. Our proofs rely on
  a connection to damped wave equations and a relationship to the
  spectral theory of Schrödinger operators.
\end{abstract}


\section{Introduction}
A typical linear kinetic equation takes the form
\begin{equation}
  \label{eq:typical-kinetic}
  \partial_t f + Tf = \sigma(x)\; C(f)
\end{equation}
for a density $f=f(t,x,v)$ at time $t$ over the phase space consisting
of a spatial position $x$ and a velocity $v$ where $T$ is a transport
operator, $\sigma$ is a spatial weight, and $C$ is a collision
operator driving the system to thermal equilibrium.

The theory of hypocoercivity, \cite{V09, DMS15}, ensures, by a variety of proofs, that the
equilibrium is reached with an exponential rate. The spectral gap
$\lambda$ limiting the decay behaves for a constant $\sigma$ typically
as indicated in \cref{fig:typical-decay-rate}.  Here we see two
distinct regimes:
\begin{enumerate}
\item For small jump rates $\sigma$ the spectral gap scales with
  $\sigma$. Here the decay is limited by the thermalisation rate of
  the velocity variable so that a faster jump rate improves the
  spectral gap.
\item For bigger jump rates $\sigma$ the spectral gap behaves like
  $\sigma^{-1}$. In this regime the decay rate is limited by the
  spatial diffusion. Here a faster decay rate means slower decay as
  the effective spatial transport decreases by the law of large
  numbers.
\end{enumerate}

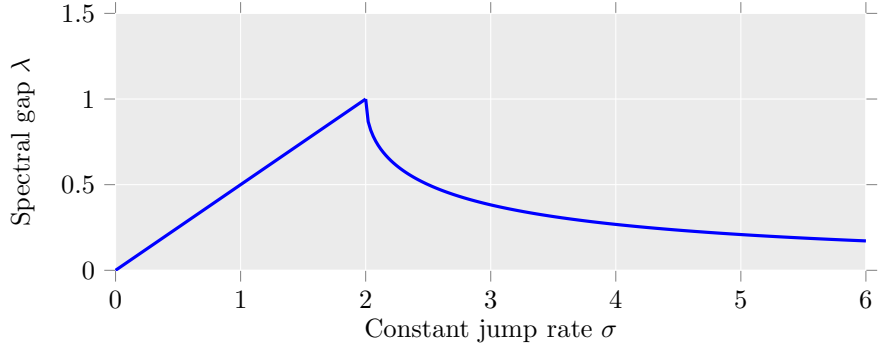
\begin{figure}[h]
  \centering
  \begin{tikzpicture}
    \begin{axis}[
      xlabel={Constant jump rate \(\sigma\)},
      xmin=0.0, xmax=6.0,
      ylabel={Spectral gap \(\lambda\)},
      ymin=0.0, ymax=1.5,
      height=5cm,width=0.7\textwidth
      ]
      \addplot [domain=0:2,very thick,blue] {x/2};
      \addplot [domain=2:6,samples=200,very thick,blue] {(x/2)-sqrt(x*x/4-1)};
    \end{axis}
  \end{tikzpicture}
  \caption{Typical decay rate depending on the noise strength. For a
    given constant $\sigma$, we plot the spectral gap for typical
    models of \eqref{eq:typical-kinetic}.}
  \label{fig:typical-decay-rate}
\end{figure}

This motivates the main question of this research.
\begin{question}
  Can we combine spatial regions of large and small jump rates in
  order to obtain a faster decay rate? More generally, what is the
  jump rate $\sigma$ in order to find the largest spectral gap, i.e.\
  the fastest decay?
\end{question}

There have been several research works which fix $\sigma$ and find
bounds on the rate of convergence to equilibrium for the system; these
works fit into the general framework of \emph{hypocoercivity}. The
goal of this work is to understand the dependence of the rate of
convergence to equilibrium on $\sigma$ by studying the optimal control
problem of finding the choice of $\sigma$ which maximises the rate. We
believe this provides another direction to understand the precise
dependence of the decay rate on the parameters in the equation.  We
believe this approach has potential in a variety of other kinetic
equations:
\begin{itemize}
\item For kinetic equations with a confining potential one could
  investgate the interplay between jump rate and the confining
  potential.
\item For equations posed on a domain with boundary one could
  investigate, in a similar way, the dependence of the rate on the shape
  of the domain and boundary conditions. This could produce results
  similar to the celebrated Faber-Krahn inequality.
\item This problem is related to the control of nuclear reactors as
  for the radiative transfer equation $\sigma$ is related to the
  presence or absence of control rods.
\item In our perturbation result we show that the optimal $\sigma$
  must occur simultaneously with a degeneracy in the eigenspace
  associated to the spectral gap eigenvectors. We believe this might
  point to connections between the optimal $\sigma$ and symmetries
  present in the equation.
\end{itemize}

Apart from applications in kinetic theory, Markov Chain Monte Carlo
(MCMC) algorithms are a main motivation. In applications of Bayesian
statistics, one needs to calculate the posterior distribution which is
given up to a normalisation factor by
\begin{equation*}
  \ee^{-\phi(x)}.
\end{equation*}
For a high-dimensional problems an explicit computation is prohibitively
expensive and a common solution is to construct a stochastic process
$Z$ which converges to the sought distribution and to sample from that
process. One such a process is a diffusion process
\begin{equation*}
  \dd X_t = -\nabla_x \phi(X_t)\, \dd t + \dd W_t.
\end{equation*}
This procedure can sometimes be slow, and Hamiltonian MCMC (HMCMC) has
been developed as a way to increase the speed of convergence of these
algorithms, see \cite{EGZ19, BEZ20} for a rigorous proof of the
increase in speed and references within on HMCMC. The strategy of
Hamiltonian Markov Chain Monte Carlo is to look at the related kinetic
equation
\begin{equation*}
  \left\{
    \begin{aligned}
      \dd X_t &= V_t\, \dd t,\\
      \dd V_t &= - \nabla_x \phi(X_t)\, \dd t + \sigma(X_t)\;
      (\dd W_t - V_t\, \dd t)
    \end{aligned}
  \right.
\end{equation*}
which has the equilibrium distribution $M(v)\, \ee^{-\phi(x)}$ for the
velocity equilibrium $M(v)$ so that the sought distribution is
obtained by the spatial distribution. Here the intuitive idea is that
the kinetic equation yields a faster transport of the distribution
over large spatial distances. The previous analyses look at the case
of constant $\sigma$ and we now ask the further question whether the
speed of convergence of these processes can be increase by making
$\sigma$ spatially dependent. This has been investigated numerically
in statistics literature, for example in \cite{GC11}, where they
propose a version of the Metropolis adjusted Langevin algorithm (MALA)
which takes into account the geometry of $\phi$.

A very simple model to study the exponential decay of kinetic
equations is the one-dimensional Goldstein-Taylor model, which is
still actively studied as a test case for hypocoercive results
\cite{arnold-einav-signorello-woehrer-2021-non-homogeneous-goldstein-taylor}, and has been studied with $\sigma$ depending on $x$ in \cite{bernard-salvarani-2013-optimal-estimate-spectral-gap-degenerate-gt}, relating it to the work \textcite{lebeau-1994-equations-des-ondes-amorties}.
It is a special case of BGK models with only two velocities \(\pm
1\). Setting \(u=(t,x) = f(t,x,+1)\) and \(v(t,x) = f(t,x,-1)\) to the
respective spatial densities, the model writes
\begin{equation}
  \label{eq:goldstein-taylor}
  \left\{
    \begin{aligned}
      \partial_t u + \partial_x u &= \frac{\sigma(x)}{2}\; (v-u), \\
      \partial_t v - \partial_x v &= \frac{\sigma(x)}{2}\; (u-v),
    \end{aligned}
  \right.
\end{equation}
where we consider the spatial variable \(x\) in the torus \(\T\) with
length \(2\pi\).

As used before \cite{kac-1974-stochastic-model-telegraphers-equation,
  bernard-salvarani-2013-optimal-estimate-spectral-gap-degenerate-gt},
the one-dimensional case has the special feature that the kinetic
equation \eqref{eq:goldstein-taylor} is equivalent to a damped wave
equation by considering
\begin{equation}
  \label{eq:wave-variables}
  \rho(t,x) := \frac{u(t,x) + v(t,x)}{\sqrt 2}
  \qquad \text{and} \qquad
  j(t,x) := \frac{u(t,x) - v(t,x)}{\sqrt 2}.
\end{equation}
Then the Goldstein-Taylor model \eqref{eq:goldstein-taylor} can be
written as
\begin{equation}
  \label{eq:goldstein-taylor-wave}
  \left\{
    \begin{aligned}
      \partial_t \rho + \partial_xj &= 0, \\
      \partial_t j + \partial_x \rho &= -\sigma(x)\; j.
    \end{aligned}
  \right.
\end{equation}

In our results we want to characterise the convergence towards the
stationary state which is in the formulation
\eqref{eq:goldstein-taylor} given by \(u=v=\const\) or in the
formulation \eqref{eq:goldstein-taylor-wave} by \(\rho=\const\) and
\(j=0\). By the conservation of the mass \(\int_\T \rho(x)\, \dd x\),
the limiting state can be characterised and using the linearity it
therefore suffices to study the perturbation from the limiting space.

Working in \(\fsL^2\) we therefore consider the evolution in the space
\begin{equation}
  \label{eq:definition-l2p}
  \fsL^2_p = \{
  (\rho,j) \in \fsL^2(\T) : \int_\T \rho(x)\, \dd x = 0
  \}
\end{equation}
with the natural norm given by
\begin{equation*}
  \| (\rho,j) \|_2^2
  = \| \rho \|_2^2 + \| j \|_2^2.
\end{equation*}

The evolution in \(\fsL^2_p\) can be understood with a semigroup with
the generator \(A_{\sigma}\) (see \cref{thm:semigroup-setup} below).
As a first result we characterise a possible jump rate
\(\sigma=\sigma(x)\) by considering perturbations.
\begin{thm}
  \label{thm:perturbation}
  Suppose \(\sigma \in \fsL^\infty(\T)\) is such that the spectral gap
  in \(\fsL^2_p\) is locally maximised. Then the spectral gap is not
  determined by a simple eigenvalue.
\end{thm}

Such a wave equation of a string has been studied in
\textcite{cox-zuazua-1994-rate-decay-damped-string} with fixed
ends. One aspect of their work is to characterise the eigenvalues
along the real axis by the spectrum of Schrödinger operators which
allows them to find the spatial damping \(\sigma\) minimising the
largest real eigenvalues, but they cannot consider the full spectral
gap. However, in our situation, and with our aims in mind, we are able
to go further by associating a different Schrödinger operator. In this
newly associated Schrödinger operator we can obtain the result by
looking at the second eigenvalue and exploiting the translation
symmetry; this is key to our bound when $\|\sigma\|_1$ is large. Due
to the different boundary values we capture the true spectral gap in
contrast to \textcite{cox-zuazua-1994-rate-decay-damped-string} where
their result for the eigenvalues along the real axis does not capture
the spectral gap. This yields the following theorem:
\begin{thm}
  \label{thm:global-optimal}
  For the Goldstein-Taylor model \eqref{eq:goldstein-taylor}, the
  largest spectral gap in \(\fsL^2_p\) is obtained with the constant
  \(\sigma=2\) giving the spectral gap \(1\).
\end{thm}

In the context of the wave equation, the corresponding question of the
spatially dependent damping has been studied before and shows that the competing
effects of the jump rate are more intricate as the result of a
constant damping in \cref{thm:global-optimal} might suggests. So it is
noted by \textcite{castro-cox-2001-achieving-arbitrarly-large-decay}
that an arbitrary large decay rate can be obtained in the case of
fixed-ends and a spatialy dependent damping diverging towards the boundary. Taking
the damping as an indicator function of a set \(\omega\) and optimising
the set \(\omega\), the competition between the effects is non-trivial
and yields in general non-existence of optimal sets
\cite{hebrard-henrott-2003-optimal-shape-stabilization-string,
  muench-pedregal-periago-2006-optimal-design-damping-set}. Phrased in
terms of the related observability condition it has been further
studied in
\cite{privat-trelat-zuazua-2013-optimal-observation-one-dimensional-wave}. The
problem has also been formulated in terms of the overall energy
\cite{cox-1998-designing-optimal-enegery-absorption} and from a
numerical side the problem is also studied by e.g.\
\cite{fahroo-ito-2006-optimal-absorption-design}. It is also studied in more complex geometries in \textcite{lebeau-1994-equations-des-ondes-amorties}.

The study of the decay rate for the presented class of systems is a
wide field ranging from works in kinetic theory
\cite{achleitner-arnold-carlen-2016-linear-hypocoercive-bgk,
  achleitner-arnold-signorello-2019-optimal-decay-estimates-model-decomposition,
  arnold-einav-signorello-woehrer-2021-non-homogeneous-goldstein-taylor}
to the wave equation
\cite{haraux-lieard-privat-2016-observability-constant-one-dimensional-wave}
to numerical methods
\cite{ammari-castro-2019-numerical-approximation-best-decay-rate}.

As a first step, we formulate in \cref{sec:spectral} the spectral
problem precisely and also shows that the spectral gap determines
the decay rate of the $L^2$ norm under the flow of our equation. Furthermore, we find eigenvalues corresponding to
the decay rate \(\|\sigma\|_1/(4\pi)\) of the velocity distribution
alone.
\begin{prop}
  \label{thm:semigroup-setup}
  For non-negative \(\sigma \in \fsL^\infty\) the closed linear
  operator \(A_\sigma\) defined by
  \begin{equation}
    \label{eq:definition-generator}
    A_\sigma
    \begin{pmatrix}
      \rho \\ j
    \end{pmatrix}
    =
    \begin{pmatrix}
      0 & -\partial_x \\
      -\partial_x & -\sigma(x)
    \end{pmatrix}
    \begin{pmatrix}
      \rho \\ j
    \end{pmatrix}
  \end{equation}
  on \(\fsL^2_p\) generates a contraction semigroup
  \((\ee^{tA_\sigma})_t\) matching the evolution
  \eqref{eq:goldstein-taylor-wave}.

  For any \(\epsilon > 0\) there exists an eigenvalue
  \(\lambda \in \C\) with
  \(|\Re \lambda - \|\sigma\|_1/(4\pi)| \le \epsilon\) and there are
  at most finitely many eigenvalues in the halfspace
  \(\{ \lambda \in \C : \Re \lambda \ge \|\sigma\|_1/(4\pi) +
  \epsilon\}\).

  If for \(a > - \| \sigma \|_1/(4\pi)\) there exists no eigenvalue
  with \(\Re \lambda \ge a\), then we have the growth bound
  \begin{equation*}
    \| \ee^{t A_{\sigma}} \|
    \lesssim \ee^{t a}
    \qquad \forall t \ge 0.
  \end{equation*}
\end{prop}
Note that the corresponding results for the wave equation with fixed
ends have been shown in
\textcite{cox-zuazua-1994-rate-decay-damped-string}.

The main implication of our study is that, for the Goldstein-Taylor
model, the rate of convergence to equilibrium cannot be improved by
making $\sigma$ depend on $x$. We believe this suggests that the rate
of convergence to equilibrium is unlikely to be increased, in kinetic
models, by adding local oscillations to the jump rate $\sigma$. We note
that this does not exclude the possiblility that spatially dependent
jump rates cannot improve the rate of convergence to equilibrium in
the presence of more complex geometries. In fact, it is proposed in
\cite{GC11} to vary $\sigma$ on large scales in a way that is
sympathetic to the confining function $\phi$. Therefore we close the
introduction by the following open question:

\begin{question}
  Does there exist a linear kinetic equation posed on $\mathbb{R}^d$
  for some $d$ with a confining potential for which a strictly faster
  rate of convergence can be achieved by allowing the collision rate
  $\sigma$ to depend on $x$ than is achieved by constant $\sigma$?
\end{question}

\section{Semigroup and spectral problem}
\label{sec:spectral}

In this section we prove \cref{thm:semigroup-setup} in two
parts. First we show that it we have a well defined semigroup for the
flow.
\begin{proof}[Proof of \cref{thm:semigroup-setup} (first part)]
  The generator \eqref{eq:definition-generator} formally gives the
  required PDE \eqref{eq:goldstein-taylor-wave}. Without the
  \(\sigma\) the solution is given explicitly by the characteristics
  defined by the transport and this explicit representation shows that
  it generates a semigroup. As \(\sigma \in \fsL^\infty\), the
  contribution of \(\sigma\) in \(A_\sigma\) is a bounded perturbation
  so that it defines a semigroup.

  The required mass conservation follows from the estimate that
  \begin{equation*}
    \frac{\dd}{\dd t} \int_\T \rho(t,x) \, \dd x
    = - \int_T \partial_x j(t,x)\, \dd x = 0
  \end{equation*}
  and the contraction property by the estimate
  \begin{equation*}
    \frac{\dd}{\dd t} \int_\T |\rho(t,x)|^2 + |j|^2 \, \dd x
    = - \int_T
    \left(
      \conj{\rho} \partial_x j + j (\partial_x \conj{\rho} + \sigma(x)
      \conj{j})
    \right) \dd x
    = - \int_T
    \sigma(x)\, |j|^2
    \dd x \le 0.
  \end{equation*}
  This gives the first part of \cref{thm:semigroup-setup}.
\end{proof}

In the spectral property the central object is the resolvent
\begin{equation}
  \label{eq:definition-resolvent}
  R(\lambda,A_\sigma) := (\lambda - A_{\sigma})^{-1}.
\end{equation}

For \((a,b) \in \fsL^2_p\) the image
\((\rho,j) = R(\lambda,A_{\sigma})(a,b)\) is the solution to
\begin{equation}
  \label{eq:resolvent-abstract}
  (\lambda-A_{\sigma})
  \begin{pmatrix}
    \rho \\ j
  \end{pmatrix}
  =
  \begin{pmatrix}
    a \\ b
  \end{pmatrix}
\end{equation}
as long it has a unique solution. This equation can be rewritten as
\begin{equation}
  \label{eq:resolvent-dx}
  \frac{\dd}{\dd x}
  \begin{pmatrix}
    \rho \\ j
  \end{pmatrix}
  +
  \begin{pmatrix}
    0 & \sigma+\lambda \\
    \lambda & 0
  \end{pmatrix}
  \begin{pmatrix}
    \rho \\ j
  \end{pmatrix}
  =
  \begin{pmatrix}
    b \\ a
  \end{pmatrix}.
\end{equation}

Associated to \eqref{eq:resolvent-dx} we define the solution operator
for the homogeneous part and an operator \(M(\sigma,\lambda)\)
yielding the solvability condition.
\begin{defn}
  \label{def:solution-dx}
  Fix \(\sigma \in \fsL^\infty\) and \(\lambda \in \C\).  For
  \(y \in \T\) and given \((\rho_\init,j_\init)\) consider the linear
  ODE
  \begin{equation}
    \label{eq:resolvent-homogeneous-dx}
    \left\{
      \begin{aligned}
        &\frac{\dd}{\dd x}
        \begin{pmatrix}
          \rho(x) \\ j(x)
        \end{pmatrix}
        +
        \begin{pmatrix}
          0 & \sigma(x)+\lambda \\
          \lambda & 0
        \end{pmatrix}
        \begin{pmatrix}
          \rho(x) \\ j(x)
        \end{pmatrix}
        = 0\\
        &\begin{pmatrix}
          \rho(y) \\ j(y)
        \end{pmatrix}
        =
        \begin{pmatrix}
          \rho_\init \\ j_\init
        \end{pmatrix}
      \end{aligned}
    \right.
  \end{equation}
  and define \(\sol{\sigma,\lambda}{y}{x}\) as solution operator so
  that
  \(\sol{\sigma,\lambda}{y}{x}((\rho_\init,j_\init)) =
  (\rho(x),j(x))\). For the resolvent define the operator
  \begin{equation}
    \label{eq:matrix-m}
    M(\sigma,\lambda) = \id - \sol{\sigma,\lambda}{0}{2\pi}.
  \end{equation}
\end{defn}

By Duhamel's principle the resolvent equation \eqref{eq:resolvent-dx}
can be written with the solution operator as
\begin{equation}
  \label{eq:resolvent-solution}
  \begin{pmatrix}
    \rho(y) \\ j(y)
  \end{pmatrix}
  = \sol{\sigma,\lambda}{0}{y}
  \begin{pmatrix}
    \rho_0 \\ j_0
  \end{pmatrix}
  +
  \int_0^y
  \sol{\sigma,\lambda}{x}{y}
  \begin{pmatrix}
    b(x) \\ a(x)
  \end{pmatrix}
  \dd x
\end{equation}
for constants \(\rho_0,j_0\). By the periodic boundary condition
\(\rho(2\pi)=\rho_0\) and \(j(2\pi)=j_0\) we must have
\begin{equation}
  \label{eq:resolvent-determine-constant}
  M(\sigma,\lambda)
  \begin{pmatrix}
    \rho_0 \\ j_0
  \end{pmatrix}
  = \int_0^{2\pi}
  \sol{\sigma,\lambda}{x}{2\pi}
  \begin{pmatrix}
    b(x) \\ a(x)
  \end{pmatrix}
  \dd x
\end{equation}
which yields a unique solution if \(M(\sigma,\lambda)\) is
invertible. If not the kernel gives an eigenvalue solving
\begin{equation}
  \label{eq:condition-eigenvector}
  \left\{
    \begin{aligned}
      \lambda \rho + \partial_x j &= 0,\\
      \partial_x \rho + (\lambda+\sigma(x)) j &= 0.
    \end{aligned}
  \right.
\end{equation}

The characterisation by the operator \(M\) corresponds to the
well-known shooting method for solving eigenvalue problems as done in
\cite{cox-zuazua-1994-rate-decay-damped-string}. For the further
analysis we note basic properties.
\begin{lemma}
  \label{thm:continuity-m-sol}
  For a fixed \(\sigma \in \fsL^\infty\) and \(x,y \in \T\), the
  solution operator \(\sol{\sigma,\lambda}{x}{y}\) and
  \(M(\sigma,\lambda)\) are analytic with respect to
  \(\lambda \in \C\).

  Moreover, the solution operator is bounded as
  \begin{equation*}
    \| \sol{\sigma,\lambda}{x}{y} \|
    \le \exp(|\Re(\lambda)|\, |y-x| + \| \sigma \|_1/2)
  \end{equation*}
  and the solution operator \(\sol{\sigma,\lambda}{x}{y}\) and
  \(M(\sigma,\lambda)\) are continuous differentiable with respect to
  \(\|\sigma\|_1\).
\end{lemma}
\begin{proof}
  The existence and analyticity  of the solution operator
  \(\sol{\sigma,\lambda}{x}{y}\) and \(M(\sigma,\lambda)\) follow from
  standard ODE theory.

  For the growth bound and the dependence with respect to
  \(\|\sigma\|_1\) note the following a priori estimate
  \begin{equation*}
    \begin{split}
      \frac{\dd}{\dd x} \frac 12 \big(|\rho|^2 + |j|^2)
      &= \conj{\rho} \partial_x \rho + j \partial_x \conj{j} \\
      &= - \conj{\rho} (\sigma+\lambda) j - j \conj{\lambda}
      \conj{\rho} \\
      &= - \conj{\rho} j (\sigma+\lambda+\conj{\lambda}) \\
      &\le \Big[2 |\Re(\lambda)| + \sigma(x) \Big]
      \frac 12 \big(|\rho|^2 + |j|^2),
    \end{split}
  \end{equation*}
  which yields the result.
\end{proof}

Hence we can precisely describe the resolvent.
\begin{lemma}
  \label{thm:resolvent-bound}
  The resolvent set consist of \(\lambda \in \C\) for which the matrix
  \(M(\sigma,\lambda)\) is invertible and the resolvent is bounded in
  the operator norm as
  \begin{equation*}
    \| R(\lambda,A) \|
    \lesssim
    \| M(\sigma,\lambda)^{-1} \|\;
    \exp\Big(4\pi|\Re(\lambda)| + \| \sigma \|_1\Big).
  \end{equation*}
  Moreover, for every \(\lambda\) in the spectrum of \(A_{\sigma}\)
  there exists at least one eigenvector.
\end{lemma}
\begin{proof}
  The result follows directly from the representation
  \eqref{eq:resolvent-solution} once we determine the constants
  \(\rho_0,j_0\) by \eqref{eq:resolvent-determine-constant} and use
  the bound from \cref{thm:continuity-m-sol}.

  If \(\lambda\) is in the spectrum of \(A_{\sigma}\), then there
  exists an element in the kernel of \(M(\sigma,\lambda)\) which
  yields an eigenvector.
\end{proof}

We now look at the asymptotic form of \(M\) as
\(|\Im \lambda|\to\infty\) over a finite range of \(\Re
\lambda\).
\begin{lemma}
  \label{thm:asymptotic-m}
  Fix a bounded interval \(I \in \R\) and
  \(\sigma \in \fsL^\infty(\T)\), then
  \begin{equation*}
    M(\sigma,\lambda)
    \to
  \begin{pmatrix}
    1-\cosh\left(\lambda\left(2\pi + \frac{\| \sigma \|_1}{2\lambda}\right)\right)
    &\sinh\left(\lambda\left(2\pi + \frac{\| \sigma \|_1}{2\lambda}\right)\right) \\
    \sinh\left(\lambda\left(2\pi + \frac{\| \sigma \|_1}{2\lambda}\right)\right)
    &1-\cosh\left(\lambda\left(2\pi + \frac{\| \sigma \|_1}{2\lambda}\right)\right)
  \end{pmatrix}
  \end{equation*}
  uniformly over \(\Re \lambda \in I\) as \(|\Im \lambda| \to \infty\).
\end{lemma}
This matches
\cite[Theorem~5.1]{cox-1998-designing-optimal-enegery-absorption}
where the result has been proven with explicit error bounds by a
series solution. For being self-contained, we give another shorter
proof.
\begin{proof}
  The idea is like in the proof of the Riemann-Lebesgue lemma that we
  can approximate \(\sigma\) by a piecewise constant function
  \(\tilde{\sigma}\), i.e.\ there exists \(x_0=0 < x_1 < x_2 < \dots <
  x_K=2\pi\) and \(\sigma_1,\sigma_2,\dots,\sigma_K\) such that
  \begin{equation*}
    \tilde{\sigma}(y) = \sigma_j \qquad \forall y \in [x_{j-1},x_j).
  \end{equation*}

  On each constant part we find the explicit solution
  \begin{equation*}
    \sol{\tilde{\sigma},\lambda}{x_{j-1}}{x_j} =
    \begin{pmatrix}
      \cosh\Big((x_j-x_{j-1}) \sqrt{\lambda(\lambda+\sigma_j)}\Big)
      & - \frac{\lambda+\sigma_j}{\lambda}
      \sinh\Big((x_j-x_{j-1}) \sqrt{\lambda(\lambda+\sigma_j)}\Big) \\
      - \frac{\lambda}{\lambda+\sigma_j}
      \sinh\Big((x_j-x_{j-1}) \sqrt{\lambda(\lambda+\sigma_j)}\Big)
      & \cosh\Big((x_j-x_{j-1}) \sqrt{\lambda(\lambda+\sigma_j)}\Big)
    \end{pmatrix}
  \end{equation*}
  and for \(|\Im \lambda| \to \infty\) this is converging uniformly to
  \begin{equation*}
    \tilde{S}_j :=
    \begin{pmatrix}
      \cosh\left(\lambda \left((x_j-x_{j-1}) + \frac{(x_j-x_{j-1})\sigma_j}{2\lambda}\right)\right)
      & - \sinh\left(\lambda \left((x_j-x_{j-1}) +
          \frac{(x_j-x_{j-1})\sigma_j}{2\lambda}\right)\right) \\
      - \sinh\left(\lambda \left((x_j-x_{j-1}) +
          \frac{(x_j-x_{j-1})\sigma_j}{2\lambda}\right)\right)
      & \cosh\left(\lambda \left((x_j-x_{j-1}) + \frac{(x_j-x_{j-1})\sigma_j}{2\lambda}\right)\right)
    \end{pmatrix}.
  \end{equation*}
  By the group property of the solution operator we find
  \begin{equation*}
    \sol{\tilde{\sigma},\lambda}{0}{2\pi}
    = \sol{\tilde{\sigma},\lambda}{x_{j-1}}{2\pi}
    \circ
    \sol{\tilde{\sigma},\lambda}{x_{j-2}}{x_{j-1}}
    \circ \dots \circ
    \sol{\tilde{\sigma},\lambda}{x_{0}}{x_{1}}
  \end{equation*}
  which therefore converges uniformly to
  \begin{equation*}
    \tilde{S}_j \circ \tilde{S}_{j-1} \circ \dots \circ \tilde{S}_{1}
    =
    \begin{pmatrix}
      \cosh\left(\lambda\left(2\pi
          + \frac{\| \tilde{\sigma} \|_1}{2\lambda}\right)\right)
      & -\sinh\left(\lambda\left(2\pi
          + \frac{\| \tilde{\sigma} \|_1}{2\lambda}\right)\right) \\
      -\sinh\left(\lambda\left(2\pi
          + \frac{\| \tilde{\sigma} \|_1}{2\lambda}\right)\right)
      & \cosh\left(\lambda\left(2\pi
          + \frac{\| \tilde{\sigma} \|_1}{2\lambda}\right)\right)
    \end{pmatrix}.
  \end{equation*}
  As we can approximate any function \(\sigma \in \fsL^\infty\)
  arbitrary well by the piecewise function in \(\fsL^1\) the result
  follows from the stability of \cref{thm:continuity-m-sol}.
\end{proof}

We can now prove the remaining parts of \cref{thm:semigroup-setup}.
\begin{proof}[Proof of \cref{thm:semigroup-setup} (remaining part)]
  The asymptotic expression of \(M\) in \cref{thm:asymptotic-m} is
  invertible except when
  \(\lambda = -\| \sigma \|_1/(4\pi) + 2\pi \ii n\) for \(n \in
  \Z\). By considering the uniform convergence over
  \(\Re \lambda \in [-\| \sigma \|_1/(4\pi)-\epsilon,-\| \sigma
  \|_1/(4\pi)+\epsilon]\) then yields the existence of root of
  limiting expression for \(\det M\) in the strip for \(|\Im \lambda|\)
  large enough. As \(\det M\) is analytic, Rouge's theorem ensures
  then the existence of a root for \(\det M\), i.e.\ an eigenvalue
  for the generator.

  As \(A_{\sigma}\) generates a contracting semigroup so that there
  are no eigenvalues \(\lambda\) with \(\Re \lambda > 0\). By the
  asymptotic expression there are no eigenvalues in
  \(-\| \sigma \|_1/(4\pi) + \epsilon\) for large enough
  \(|\Im \lambda|\). As \(M\) is analytic there can be only finitely
  many eigenvalues in the remaining bounded region.

  For the last part assume \(a > -\| \sigma \|_1/(4\pi)\) such that
  there is no eigenvalue \(\lambda\) with \(\Re \lambda \ge a\).  For
  \(\Re \lambda > 1\) we can use the fact that \((\ee^{tA_\sigma})_t\)
  is a contraction semigroup to find a uniform bound on the resolvent
  \(R(\lambda,A_\sigma)\).  By the asymptotic expression of
  \cref{thm:asymptotic-m} we have a uniform bound of
  \(\| M(\sigma,\lambda)^{-1} \|\) for large enough \(|\Im \lambda|\)
  and \(\Re \lambda \in [a,1]\).  As there are no eigenvalue with real
  part equal to \(a\) and \(M\) is continuous, this shows
  \begin{equation*}
    \sup_{\Re \lambda \in [-a,1]} \| M(\sigma,a+\ii b)^{-1} \| < \infty.
  \end{equation*}
  By \cref{thm:resolvent-bound} this shows the same bound for the
  resolvent. As our function space \(\fsL^2_p\) is a Hilbert space, we
  can thus apply Gearhart-Prüss-Greiner theorem \cite[Thm 1.11 in
  Chapter V]{engel-nagel-2000-one-parameter-semigroups} to find the
  claimed growth bound.
\end{proof}

\begin{remark}
  An alternative for using the Gearhart-Prüss-Greiner theorem is an
  adaptation of the theory of positive semigroups as done in
  \textcite{bernard-salvarani-2013-degenerate-linear-boltzmann}.
  \Textcite{cox-zuazua-1994-rate-decay-damped-string} establish the
  growth bound by studying the eigenvector system in more detail.
\end{remark}

We close this section noting the eigenvalues for the case that
\(\sigma\) is constant.
\begin{lemma}
  \label{thm:spectrum-constant-sigma}
  Suppose that \(\sigma\) is constant. Then the spectrum in
  \(\fsL^2_p\) consists of
  \begin{equation*}
    -\frac{\| \sigma \|_1}{4\pi} \pm
    \sqrt{\frac{\|\sigma\|_1^2}{(4\pi)^2} - n^2}
    \qquad \text{for } n=1,2,3,\dots
  \end{equation*}
  and
  \begin{equation*}
    - \frac{\| \sigma \|_1}{2\pi}.
  \end{equation*}
\end{lemma}
\begin{proof}
  In this setting the spatial Fourier modes decouple and the result
  follows directly by solving the eigenvalue problem for each mode,
  where we exclude the stationary state as it is done in \(\fsL^2_p\).
  Alternatively, we can use the explicit expression of \(M\) for a
  constant \(\sigma\) as in the proof of \cref{thm:asymptotic-m}.
\end{proof}

For a detailed study of the decay behaviour in the case of constant
\(\sigma\) we refer to
\textcite{achleitner-arnold-signorello-2019-optimal-decay-estimates-model-decomposition}.

\section{Perturbation analysis}
\label{sec:perturbation}

Given some \(\sigma_0 \in \fsL^\infty\) and a perturbation
\(\eta \in \fsL^\infty\) we can compute how the eigenvalues of
\(A_{\sigma_0+\epsilon\eta}\) are changing for varying
\(\epsilon\). The eigenvalues \((\rho_{\epsilon},j_{\epsilon})\)
satisfy the equation
\begin{equation*}
  (\lambda_{\epsilon} - A_{\sigma_0+\epsilon\eta})
  \begin{pmatrix}
    \rho_{\epsilon} \\ j_{\epsilon}
  \end{pmatrix}
  = 0
\end{equation*}
and formally taking the derivative with respect to \(\epsilon\) yields
\begin{equation*}
  (\lambda_{\epsilon}' - A_{\sigma_0+\epsilon\eta}')
  \begin{pmatrix}
    \rho_{\epsilon} \\ j_{\epsilon}
  \end{pmatrix}
  +
  (\lambda_{\epsilon} - A_{\sigma_0+\epsilon\eta})
  \begin{pmatrix}
    \rho_{\epsilon}' \\ j_{\epsilon}'
  \end{pmatrix}
  = 0.
\end{equation*}
Testing with a suitable adjoint \((\rho_{\epsilon}^*,j_{\epsilon}^*)\)
yields
\begin{equation*}
  \lambda_{\epsilon}'\;
  \sip{
    \begin{pmatrix}
      \rho_{\epsilon}^* \\ j_{\epsilon}^*
    \end{pmatrix}
  }
  {
    \begin{pmatrix}
      \rho_{\epsilon} \\ j_{\epsilon}
    \end{pmatrix}
  }
  =
  \sip{
    \begin{pmatrix}
      \rho_{\epsilon}^* \\ j_{\epsilon}^*
    \end{pmatrix}
  }
  {
    A_{\sigma+\epsilon\eta}'
    \begin{pmatrix}
      \rho_{\epsilon} \\ j_{\epsilon}
    \end{pmatrix}
  }.
\end{equation*}
Hence assuming that both inner products are non-zero we find that
formally \(\lambda_{\epsilon}\) is differentiable with a non-zero
derivative and thus by varying \(\epsilon\) we can change the
eigenvalue and improve the spectral gap if the spectral gap is
determined by a single eigenvalue.

The key point for this perturbation analysis is to understand when the
weight in front of \(\lambda'_{\epsilon}\) can be vanishing. By using
the formulation with respect to the operator \(M(\sigma,\lambda)\)
from \cref{def:solution-dx} we can show that for a simple eigenvalue
this prefactor never vanishes.

As a first step we sharpen the condition for the spectrum from
\cref{thm:resolvent-bound}.
\begin{lemma}
  \label{thm:characterisation-simple-eigenvector}
 Suppose $\lambda_0$ is seperated from the spectrum, then $\lambda_0$ is determined by a simple eigenvalue if and only if $\det(M, \sigma, \lambda)$ has a simple root.
\end{lemma}
\begin{proof}
  Recall the formula \eqref{eq:resolvent-solution} and
  \eqref{eq:resolvent-determine-constant} for the resolvent map. The
  inverse of \(M\) can be written as product of \(\det(M)^{-1}\) and
  the adjugate. The spectral projection \(\Pi_{\lambda_0}\) to
  \(\lambda_0\) is given by the contour integral
  \[
    \int_{C} R(\lambda, \sigma) \mathrm{d}\lambda,
  \]
  where $C$ is a simple curve separating $\lambda_0$ from the rest of
  the specturum. This integral can be computed using the Laurent
  series. This gives a one-dimensional image if and only if
  \(\det(M)\) has a simple root.
\end{proof}
\begin{remark}
  The fact that we have only finitely many eigenvalues in any strip
  $a \leq \Re (\lambda) \leq b$ with \(a > -\| \sigma \|_1/(4\pi)\)
  means that every eigenvalue is separated from the rest of the
  spectrum.
\end{remark}

This allows us to determine \(M\) around an eigenvector.
\begin{lemma}
  \label{thm:tilde-m-simple-eigenvector}
  Suppose that \(\lambda_0 \in \C\) is a simple eigenvalue. Then let
  \((\rho_0,j_0)\) be in \(\ker M(\sigma,\lambda_0)\) and normalised
  to \(\|(\rho_0,j_0)\|_2=1\). In the basis
  \begin{equation*}
    V_1 =
    \begin{pmatrix}
      \rho_0 \\ j_0
    \end{pmatrix}
    \qquad\text{and}\qquad
    V_2 =
    \begin{pmatrix}
      -\conj{j_0} \\ \conj{\rho_0}
    \end{pmatrix}
  \end{equation*}
  the operator \(M(\sigma,\lambda)\) takes the form
  \begin{equation*}
    \tilde{M}(\sigma,\lambda_0+\delta\lambda) :=
    \begin{pmatrix}
      O(\delta \lambda)
      & b + O(\delta \lambda) \\
      c \; \delta \lambda + O((\delta \lambda)^2)
      & O(\delta \lambda)
    \end{pmatrix}
  \end{equation*}
  for small \(\delta \lambda\) and constants \(b,c \not = 0\).
\end{lemma}
\begin{proof}
  By construction \(V_1\) and \(V_2\) form an orthonormal basis. Using
  the solution operator \(\sol{\sigma,\lambda_0}{\cdot}{\cdot}\) from
  \cref{def:solution-dx} we define
  \begin{equation}
    \label{eq:rho-j-1}
    \begin{pmatrix}
      \rho_1(x) \\ j_1(x)
    \end{pmatrix}
    = \sol{\sigma,\lambda_0}{0}{x}
    \begin{pmatrix}
      \rho_0 \\ j_0
    \end{pmatrix}
  \end{equation}
  and
  \begin{equation}
    \label{eq:rho-j-2}
    \begin{pmatrix}
      \rho_2(x) \\ j_2(x)
    \end{pmatrix}
    = \sol{\sigma,\lambda_0}{0}{x}
    \begin{pmatrix}
      -\conj{j_0} \\ \conj{\rho_0}
    \end{pmatrix}.
  \end{equation}
  The idea of the Wronskian is to consider \(\rho_1(x)j_2(x)-\rho_2(x)j_1(x)\)
  and by \eqref{eq:resolvent-homogeneous-dx} we find
  \begin{equation*}
    \frac{\dd}{\dd x}
    \Big(
      \rho_1(x)j_2(x)-\rho_2(x)j_1(x)
    \Big)
    = 0.
  \end{equation*}
  As \((\rho_0,j_0) \in \ker M(\sigma,\lambda_0)\) we have that
  \((\rho_1(2\pi),j_1(2\pi))=(\rho_0,j_0)\). For \(V_2\) we find with
  the Wronskian
  \begin{equation*}
    \begin{split}
      \sip{
        \begin{pmatrix}
          \rho_2(2\pi) \\ j_2(2\pi)
        \end{pmatrix}
      }{V_2}
      &= \rho_0\, j_2(2\pi) - j_0\, \rho_2(2\pi)
      = \rho_1(2\pi)\, j_2(2\pi) - j_1(2\pi)\, \rho_2(2\pi) \\
      &= \rho_1(0) j_2(0) - \rho_2(0) j_1(0) = |\rho_0|^2 + |j_0|^2 = 1.
    \end{split}
  \end{equation*}
  Hence we have found that \(\ip{M(\sigma,\lambda_0)V_2}{V_2}=0\).
  Further recalling that \(V_1 \in \ker M(\sigma,\lambda_0)\) and that
  \(M\) and \(\tilde{M}\) are analytic with respect to \(\lambda\)
  shows that
  \begin{equation*}
    \tilde{M}(\sigma,\lambda_0+\delta\lambda) =
    \begin{pmatrix}
      O(\delta \lambda)
      & b + O(\delta \lambda) \\
      c \; \delta \lambda + O((\delta \lambda)^2)
      & O(\delta \lambda)
    \end{pmatrix}
  \end{equation*}
  for some constants \(b,c \in \C\). By
  \cref{thm:characterisation-simple-eigenvector}, the determinant
  \(M(\sigma,\lambda)=\tilde{M}(\sigma,\lambda)\) must have a single
  root at \(\lambda_0\) so that \(b,c \not = 0\).
\end{proof}
We can now prove the perturbation result.
\begin{proof}[Proof of \cref{thm:perturbation}]
  We argue by contradiction and assume that for a
  \(\sigma \in \fsL^\infty\) the spectral gap is determined by the
  simple eigenvalue \(\lambda_0\).

  For a perturbation \(\eta \in \fsL^\infty\), we then find for
  \(M(\sigma,\lambda)\) around \(\lambda_0\) in the form of
  \(\tilde{M}(\sigma,\lambda)\) defined in
  \cref{thm:tilde-m-simple-eigenvector} that
  \begin{equation*}
    \tilde{M}(\sigma+\epsilon \eta,\lambda_0+\delta\lambda) =
    \begin{pmatrix}
      O(\epsilon,\delta \lambda)
      & b + O(\epsilon,\delta \lambda) \\
      c \; \delta \lambda + d \epsilon + O((\epsilon,\delta \lambda)^2)
      & O(\epsilon,\delta \lambda)
    \end{pmatrix}.
  \end{equation*}
  By Duhamel's principle we have that
  \[
    \sol{\sigma + \epsilon \eta ,\lambda_0}{0}{x} =
    \sol{\sigma,\lambda_0}{0}{x} + \int_0^x
    \sol{\sigma,\lambda_0}{y}{x}\epsilon \eta(y)
    \begin{pmatrix}
      0 & -1 \\ 0 & 0
    \end{pmatrix}
    \sol{\sigma+\epsilon\eta,\lambda_0}{0}{y}\, \dd y.
  \]
  And the bottom left hand term of the matrix
  $\sol{\sigma+\epsilon\eta,\lambda_0}{0}{2\pi}$ will be
  $\ip{ \sol{\sigma + \epsilon \eta ,\lambda_0}{0}{2\pi}V_1}{V_2}$.
  Repeating the Wronskian argument as in the proof of
  \cref{thm:tilde-m-simple-eigenvector} shows that
  \begin{equation*}
    d
    = \int_0^{2\pi} \eta(y)
    \ip{\sol{\sigma,\lambda_0}{x}{2\pi}
      \begin{pmatrix}
        0 & -1 \\ 0 & 0
      \end{pmatrix}
      \sol{\sigma,\lambda_0}{0}{x}}{V_2}\, \dd y
    = \int_0^{2\pi}
    \begin{pmatrix}
      -j_0 \\ \rho_0
    \end{pmatrix}
    \cdot
    \sol{\sigma,\lambda_0}{y}{2\pi}
    \begin{pmatrix}
      -j_1(y) \\ 0
    \end{pmatrix}
    \dd y
    =  \int_0^{2\pi} j_1(y)^2\, \eta(y)\, \dd y
  \end{equation*}
  with \(j_1(y)\) from \eqref{eq:rho-j-1}\footnote{One could also
    prove \(c=\int_0^{2\pi} (j_1(y)^2 - \rho_1(y)^2)\, \dd y\).}. As
  the eigenvalue is determined by \(\det \tilde{M}\) this shows that
  the eigenvalue behaves as
  \begin{equation*}
    \lambda_0 - \frac{d}{c} \epsilon + O(\epsilon^2).
  \end{equation*}
  Hence if we can find some \(\eta\) such that \(\Re(d/c) \not =0\),
  we can choose a small \(\epsilon \in \R\) so that
  \(\sigma+\epsilon\eta\) would have a bigger spectral gap. This shows
  the result if \(\sigma+\epsilon\eta\) is a valid jump rate, i.e.\
  non-negative.

  As \(\sigma\) is a non-trivial jump rate, it is strictly positive in
  a subset \(I\) of positive measure. If \(j_1(y)^2\) has not constant
  complex phase, we can always construct \(\eta \in \fsL^\infty\) with
  support in \(I\) such that \(\Re(d/c) \not = 0\) and thus
  \(\sigma+\epsilon\eta\) yields a valid perturbation of the jump
  rate. In the case that \(\lambda \in \R\), the real and imaginary
  part decouple and the constants \(d\) and \(c\) must be real. In the
  other case the system is translation invariant and thus the spectral
  gap cannot determined by a simple eigenvalue.
\end{proof}

\section{Global optimum by associated Schrödinger equation}
\label{sec:schroedinger}

The eigenvector equation \eqref{eq:condition-eigenvector} can be
written as the following second order equation
\begin{equation}
  \label{eq:eigenvalue-second-order-j}
  - \partial_x^2 j + \lambda (\lambda+\sigma) j = 0.
\end{equation}
For a fixed \(\lambda \in \R\) we then consider the Hamiltonian
\(H_{\sigma,\lambda}\) by
\begin{equation}
  \label{eq:definition-hamiltion-sigma-lambda}
  H_{\sigma,\lambda} j
  = -\partial_x^2 j + \lambda (\lambda+\sigma(x))\, j.
\end{equation}
By the construction, our generator \(A_{\sigma}\) has an eigenvector
with eigenvalue \(\lambda\) if \(H_{\sigma,\lambda}\) has a zero
eigenvector. By looking at the evolution of the spectrum, similar to
\textcite{cox-zuazua-1994-rate-decay-damped-string}, we can show a
slowly decaying eigenvector in the diffusive regime.
\begin{prop}
  \label{thm:schroedinger}
  Suppose that \(\|\sigma\|_1 > 4\pi\) and let
  \begin{equation*}
    \lambda_s = -\frac{\|\sigma\|_1}{4\pi}
    + \sqrt{\left(\frac{\|\sigma\|_1}{4\pi}\right)^2 -1}.
  \end{equation*}
  Then there exists a \(\lambda \in [\lambda_s,0)\) such that
  \(\lambda\) is an eigenvalue of the generator \(A_{\sigma}\) from
  \eqref{eq:definition-generator} of the Goldstein-Taylor system
  \eqref{eq:goldstein-taylor-wave}.
\end{prop}
\begin{proof}
  The Hamiltonian \(H_{\sigma,\lambda}\) from
  \eqref{eq:definition-hamiltion-sigma-lambda} is for
  \(\lambda \in \R\) a self-adjoint operator and has real eigenvalues
  \(\mu_1,\mu_2,\dots\) (chosen in increasing order) converging to
  \(\infty\).

  For \(\lambda=0\), the Hamiltonian \(H_{\sigma,\lambda}\) is just
  the Laplacian so that \(\mu_1 = 0\) and \(\mu_2 = 1\). By
  considering the perturbation around \(\lambda=0\), we find for a
  small \(|\lambda_s| > \epsilon > 0\) that
  for \(\lambda = 0 - \epsilon\) that \(\mu_1 < 0\) and \(\mu_2 > 0\).

  For the given \(\sigma\) find a shift \(\phi\) such that
  \begin{equation*}
    \int_0^{2\pi} \sigma(x) \cos(2(x-\phi))\, \dd x = 0
  \end{equation*}
  and consider the two test functions
  \begin{equation*}
    j_1(x) = 1
  \end{equation*}
  and
  \begin{equation*}
    j_2(x) = \sin(x-\phi)
  \end{equation*}
  which are linearly independent.

  For these test functions we find
  \begin{equation*}
    \ip{j_1}{H_{\sigma,\lambda_s}j_1} =
    \lambda_s (2\pi \lambda_s + 2 \|\sigma\|_1)
    < 0
  \end{equation*}
  and
  \begin{equation*}
    \ip{j_2}{H_{\sigma,\lambda_s}j_2} =
    \frac{2\pi}{2}
    \left(
      1 + \lambda_s
      \left(\lambda_s + \frac{\| \sigma \|_1}{2\pi}\right)
    \right)
    \le 0.
  \end{equation*}

  Hence we must have \(\mu_2(\lambda_s) \le 0\). As
  \(\mu_2(0-\epsilon) > 0\) this means that there exist some
  \(\lambda \in [\lambda_s,-\epsilon)\) such that \(\mu_2 = 0\) and by
  the previous discussion this means that there exist an eigenvalue
  for the generator \(A_{\sigma}\).
\end{proof}

We can now conclude that the constant \(\sigma=2\) yields the fastest
decay rate.
\begin{proof}[Proof of \cref{thm:global-optimal}]
  By \cref{thm:spectrum-constant-sigma} the choice \(\sigma=2\) yields
  the spectral gap \(1\).

  For another \(\sigma\) with \(\| \sigma \|_1 \le 4\pi\), the bound
  from the velocity relaxation in \cref{thm:semigroup-setup} shows
  that the spectral gap is not bigger.

  For another \(\sigma\) with \(\| \sigma \|_1 > 4\pi\), we can apply
  \cref{thm:schroedinger} to show that the spectral gap needs to be
  worse than the constant case \(\sigma=2\).
\end{proof}

\section*{Acknowledgements}

The authors gratefully acknowledge the support of the Hausdorff
Research Institute for Mathematics (Bonn), through the Junior
Trimester Program on Kinetic Theory. The authors thank Clément
Mouhot, Laurent Desvillettes, Iván Moyano for the discussions on this
project.

\setlength\bibitemsep{0pt}
\renewcommand*{\bibfont}{\small}
\printbibliography
\end{document}